\documentclass{article}

\usepackage[text={5.7in, 7.9in},centering]{geometry}
\usepackage{amsmath,amsthm,xcolor,cite}

\usepackage{bm,amssymb,mathrsfs,url,units,dsfont}
\usepackage{setspace}
\usepackage{bbold}
\usepackage{graphicx,xcolor,cite,mathtools}
\usepackage{amssymb,amsmath,amsthm,mathrsfs,paralist,bm,esint,setspace}
\usepackage{stackengine}

\RequirePackage[colorlinks,citecolor=blue,urlcolor=blue]{hyperref}

\DeclareMathOperator\Corr{Corr}

\DeclareMathOperator\Dimh{Dim_{{\rm H}}}

\renewcommand{\P}{\mathds{P}}
\newcommand{\E}{\mathds{E}}
\newcommand{\R}{\mathds{R}}
\newcommand{\Z}{\mathds{Z}}

\renewcommand{\H}{\mathcal{H}}

\newcommand{\vj}{\Vec{j}}

\newtheorem{stat}{Statement}[section]
\newtheorem{proposition}[stat]{Proposition}

\newtheorem{theorem}[stat]{Theorem}
\newtheorem{lemma}[stat]{Lemma}
\theoremstyle{definition}
\newtheorem{definition}[stat]{Definition}
\newtheorem{remark}[stat]{Remark}

\numberwithin{equation}{section}

\begin{document}

\title{Macroscopic multi-fractality of Gaussian random fields and linear SPDEs with colored noise %
	\thanks{Research supported by the NRF (National Research Foundation of Korea) grants 2019R1A5A1028324 and
2020R1A2C4002077.}
}
\author{
	Jaeyun Yi\\POSTECH
}

\date{\today}

\maketitle

\begin{abstract}
We consider the linear stochastic heat and wave equations with generalized Gaussian noise that is white in time and spatially correlated. Under the assumption that the homogeneous spatial correlation $f$ satisfies some mild conditions, we show that the solutions to the linear stochastic heat and wave equations exhibit tall peaks in macroscopic scales, which means they are macroscopically multi-fractal. We compute the macroscopic Hausdorff dimension of the peaks for Gaussian random fields with vanishing correlation and then apply this result to the solution of the linear stochastic heat and wave equations. We also study the spatio-temporal multi-fractality of the linear stochastic heat and wave equations. Our result is an extension of Khoshnevisan, Kim, and Xiao \cite{KKX,KKX2} and Kim \cite{K} to a more general class of the linear stochastic partial differential equations and Gaussian random fields.  \\

	
	\end{abstract}

\section{Introduction and main results}\label{sec1}
In this paper, we investigate the following linear stochastic heat equation with fractional Laplacian
\begin{equation}\tag{SHE}\label{SHE}
     \begin{cases} \frac{\partial}{\partial t }Z^H(t,x) =-\left(-\Delta \right)^{\alpha/2} Z^H(t,x) + \dot{F}(t,x), \quad t>0, x\in \R^d,\\
   Z^H(0,x) = 0, \quad x\in \R^d,
   \end{cases}
\end{equation}
where $\dot{F}$ is a white in time and spatially correlated (or colored) noise. In other words, $\dot{F}$ is a spatially homogeneous generalized Gaussian random field on $\R_+ \times \R^d $ whose mean is zero and covariance is given by 
\begin{equation}\label{covF}
    \E \left[  \dot{F}(t,x)\dot{F}(s,y) \right] = \delta(t-s)f( x-y), 
\end{equation}
where $f$ is a \textit{correlation measure}, i.e., a nonnegative and nonnegative definite tempered Borel measure on $\R^d$. $-\left(-\Delta \right)^{\alpha/2}$ denotes the fractional Laplacian of order $\alpha\in (0,2]$ acting on the space variable in $\R^d$. In probabilistic terminology,  $-\left(-\Delta \right)^{\alpha/2}$ is the infinitesimal generator of an isotropic $\alpha$-stable process $\{X_t\}_{t\geq0}$ in $\R^d$ whose Fourier transform is normalized so that 
$$
\E\left[ \exp \left(i \xi \cdot X_t\right) \right] = \exp \left(-t \lVert \xi \rVert^\alpha \right).
$$
We note that $-\left(-\Delta \right)^{\alpha/2}$ is $\Delta$, Laplacian, when $\alpha=2$. 

Let $\hat{f}$ be the Fourier transform of $f$, that is, $\hat{f}(\xi) := \int_{\R^d} f(x)\exp(-i \xi \cdot x) dx $. By the theory of Bochner, $\hat{f}$ is a nonnegative tempered Borel measure. Dalang \cite{Dalang} showed that with the condition
\begin{equation}\label{dalang}
      \int_{\R^d} \frac{\hat{f}(d\xi)}{1+\lVert \xi \rVert^\alpha} < \infty,
\end{equation}
the linear stochastic heat equation \eqref{SHE} has the stationary \textit{random field solution}
\begin{equation}
    Z^H(t,x) = \int_0^t\int_{\R^d} p^H_{t-s}(x-y) \, F(dsdy), \quad t>0,x\in \R^d, 
\end{equation}
where $p^H_t(x)$ is the transition density of the isotropic $\alpha$-stable process $X_t$. Moreover, it is known that the following reinforced condition 
 \begin{equation}\label{contiH}
     \int_{\R^d} \frac{\hat{f}(d\xi)}{(1+\lVert \xi \rVert^\alpha)^\eta }  < \infty, \quad \text{for some }\eta \in [0,1). 
\end{equation}
implies the almost sure H\"older continuity of $\{ Z^H(t,x) \} _{t\geq 0, x \in \R^d}$ (see \cite{BEM}). To avoid the triviality, throughout this paper, we always assume that
\begin{equation}\label{nondegeneracy}
    f(\R^d)>0,
\end{equation}
without further mention (see the proof of Theorem \ref{thmheat} below).

We are interested in the \textit{macroscopic multi-fractal} property of the solution $\{Z^H(t,x) \}_{t\geq 0, x\in \R^d}$. Khoshnevisan, Kim, and Xiao \cite{KKX} introduced a mathematical definition of the macroscopic multi-fractality (see \cite[Definition 1.1]{KKX}). Among other things, they showed that the solution to \eqref{SHE} with the correlation measure $f= \delta$\footnote{Dirac measure at $0$.} (i.e., $\dot{F}$ is space-time white noise) is multi-fractal in macroscopic  scales. To be precise, let us consider the random set
\begin{equation} \label{setheat}
    \mathcal{H}_t(\gamma) : = \left\{x\in \R^d : \lVert x \rVert > e, Z^H(t,x) \geq \sqrt {2\gamma  v^H(t) \log \lVert x \rVert  }  \right\},
\end{equation}
where $\gamma>0$ is a fixed constant and $v^H(t)$ is the variance $\E[Z^H(t,x)^2]$ for $t>0$, which is independent of $x$ due to the stationarity. We can regard the set $\mathcal{H}_t(\gamma)$ as a set of spatial tall peaks of the solution $Z^H(t,x)$ at each fixed $t>0$. $\gamma$ is a \textit{scale parameter} which scales the heights of the peaks given by the \textit{gauge function} $\sqrt{ \log \lVert x \rVert }$ (see \cite[Introduction]{KKX}). Let $\Dimh$ denote the macroscopic Hausdorff dimension introduced by Barlow and Taylor \cite{BT1,BT2}. Khoshnevisan et al. \cite{KKX} proved that, in the case of $\alpha=2$, 
\begin{equation}
    \Dimh ( \mathcal{H}_t(\gamma) ) = (1-\gamma) \vee 0.\footnote{For all reals $a$ and $b$, $a\vee b$ denotes $\max\{a,b\}$. Similarly, $a \wedge b := \min \{a,b\}$.}
\end{equation}
This implies that for infinitely many different length scales $\gamma$, the set $ \mathcal{H}_t(\gamma)$ has nontrivial distinct macroscopic Hausdorff dimensions. Hence, we can say that the linear stochastic heat equation \eqref{SHE} is (macroscopically) multi-fractal. In a recent paper \cite{K}, Kim extended the above result to the case of the linear stochastic heat equation with fractional Laplacian and colored noise of Riesz type, i.e., $\dot{F}$ has the covariance \eqref{covF} with the spatial correlation $f$ given by Riesz kernel \begin{equation}\label{Riesz}
    f(x):= c\lVert x \rVert ^{-\beta},\quad 0<\beta < \alpha \wedge d, 
\end{equation} for some constant $c=c(\beta,d)$. In this case, $\Dimh ( \mathcal{H}_t(\gamma) ) = (d-\gamma) \vee 0$ holds almost surely, with $v^H(t) = c_{\alpha,\beta,d} t^{(\alpha-\beta)/\alpha}$ for some constant $ c_{\alpha,\beta,d} >0$.

Our main objective is to extend the results of \cite{KKX} and \cite{K} in the previous paragraph to a more general class of linear stochastic partial differential equations. We show that the linear stochastic heat and wave equations with colored noise $\dot{F}$ whose corresponding correlation measure $f$ vanishes at infinity, i.e., $\lim_{\lVert x \rVert\rightarrow \infty}f( x ) = 0$, are macroscopically multi-fractal if $f$ is a genuine function. In order to deal with the case that $f$ is a measure, we also present some suitable conditions (see the conditions (ii) and (iii) in Theorem \ref{thmheat} below) for $\hat{f}$. 

Our result covers the cases of $f=\delta$, the Riesz kernel \eqref{Riesz}, and $f(x) = O(1/\log \lVert x \rVert)$. We now state the main results of this paper. 
\begin{theorem}\label{thmheat}
 Let $f$ be a correlation measure which satisfies \eqref{dalang} and \eqref{contiH}. Suppose that $f$ satisfies one of the following conditions:
 \begin{itemize}
    \item[(i)]f is a function such that  $\lim_{\lVert x \rVert \rightarrow\infty} f(x) = 0$
 \item[(ii)]$\hat{f}$ is a function, i.e., f has a spectral density $\hat{f}(\xi)$,
 \item[(iii)] $\hat{f}$ satisfies \begin{equation}
             \lim_{\lVert z \rVert \rightarrow\infty}  \int_{\R^d}\frac{e^{i \xi \cdot z} \hat{f}(d\xi) }{1+ \lVert \xi \rVert ^\alpha} =0.
 \end{equation}
 \end{itemize}
Then, for each $t>0$, we have 
\begin{equation}\label{dimh}
    \Dimh{(\mathcal{H}_t(\gamma))} = (d-\gamma) \vee 0, \qquad \text{a.s.}
\end{equation}
\end{theorem}
\begin{remark}
The condition (iii) is an equivalent condition for \eqref{SHE} to be weakly mixing. See Corollary 9.1 of Chen, Khoshevisan, Nualart and Fu \cite{CKNP}. 
One can easily see that the condition (ii) implies (iii) in Theorem \ref{thmheat} using the Riemann-Lebesgue lemma.

\end{remark}

We now consider the linear stochastic wave equation with the same assumptions on $f$ as in Theorem \ref{thmheat}.
    \begin{equation}\tag{SWE}\label{SWE}
   \begin{cases} \frac{\partial^2}{\partial t^2 }Z^W(t,x) =\Delta Z^W(t,x) + \dot{F}(t,x), \quad t>0, x\in \R^d,\\
   Z^W(0,x) = 0,  \quad\frac{\partial}{\partial t }Z^W(0,x) =0 ,\quad x\in \R^d,
   \end{cases}
\end{equation}
where $d \in \{1,2,3\}$. The theory of Dalang  \cite{Dalang} shows that the linear stochastic wave equation \eqref{SWE} has the stationary random field solution 
\begin{equation}
     Z^W(t,x) = \int_0^t\int_{\R^d} p^W_{t-s}(x-y) \, F(dsdy), \quad t>0, x\in \R^d, 
\end{equation}
where $p^W_t(x)$ is Green's function of the wave operator $\frac{\partial^2}{\partial t^2 }-\Delta$ (see \cite{minicourse}).
Let us consider the random set 
\begin{equation}\label{setwave}
    \mathcal{W}_t(\gamma) : = \left\{x\in \R^d : \lVert x \rVert > e, Z^W(t,x) \geq \sqrt {2\gamma  v^W(t) \log \lVert x \rVert  }  \right\},
\end{equation}
where $\gamma>0$ and $v^W(t) = \E\left[Z^W(t,x)^2 \right]$.
We can deduce an analogous result to Theorem \ref{thmheat}.
\begin{theorem} \label{thmwave}
 Let $f$ be a correlation measure which satisfies \eqref{dalang} and \eqref{contiH}. Suppose that $f$ satisfies one of the following conditions:
 \begin{itemize}
 \item[(i)] f is a function such that  $\lim_{\lVert x \rVert \rightarrow\infty} f(x) = 0$,
 \item[(ii)] $\hat{f}$ is a function, i.e., f has a spectral density,
 \item[(iii)] $\hat{f}$ satisfies \begin{align}
        \lim_{\lVert z \rVert \rightarrow\infty}  \int_{\R^d}\frac{e^{i \xi \cdot z} \hat{f}(d\xi) }{1+ \lVert \xi \rVert ^2} =0  
     \end{align}
 \end{itemize}
 Then, for each $t>0$, we have 
 \begin{equation}\label{dimw}
    \Dimh{(\mathcal{W}_t(\gamma))} = (d-\gamma) \vee 0, \qquad \text{a.s.}
\end{equation}
\end{theorem}
We point out that the proofs of Theorem \ref{thmheat} and \ref{thmwave} use a different method from the one used in \cite{KKX} and \cite{K}. For the lower bound of macroscopic Hausdorff dimension, the authors therein used a localization argument (Section $6$ in \cite{KKX}) or Berman's theorem (Section $3$ in \cite{K}) to construct independent random variables which are close to the original solutions. But these methods can be applied only to the case of $f=\delta$ or polynomially decaying noise (e.g., \eqref{Riesz}). Our proof strongly utilizes the fact that the solutions of \eqref{SHE} and \eqref{SWE} are Gaussian random fields (see the proofs of Theorems \ref{thmheat} and \ref{thmwave} below). Using this fact, we employ the result of Lopes \cite{Lopes} that gives an estimate for the lower tail probability of maximum of dependent Gaussian random variables.  Our result below (Theorem \ref{thmGaussian}) shows that if the correlation of a Gaussian random field vanishes at infinity, then it behaves like the infinite sum of independent random variables, which develops spatial tall peaks as a multi-fractal. Now we state the theorem for multi-fractality of Gaussian random fields with vanishing correlation. As we will see later (see Section \ref{sec4}), Theorem \ref{thmheat} and \ref{thmwave} can be obtained as corollaries of Theorem \ref{thmGaussian}.


\begin{theorem}\label{thmGaussian}
Let $\{Z(t) \}_{t\in \R^d } $ be a continuous stationary Gaussian random field with mean zero and variance one. Suppose that the correlation $\Corr \left( Z(t) , Z(s)\right)=: \rho ( \lVert t-s \rVert  ) \rightarrow 0 $ monotonically as $\lVert t-s \rVert \rightarrow \infty$. Let us define the random set
\begin{equation}\label{setGaussian}
    \mathcal{Z}(\gamma) := \left\{ t\in \R^d  : \lVert t \rVert \geq e , Z(t) \geq  \sqrt{2\gamma \log \lVert t\rVert }\right\},
\end{equation} for every $\gamma>0$.
Then we have 
\begin{equation} \label{dimgaussian}
    \Dimh \left(\mathcal{Z}(\gamma) \right) = (d - \gamma) \vee 0, \qquad \text{a.s.,}
\end{equation}
for all $\gamma>0$.
\end{theorem} 


We finish our introduction with the theorems for the \textit{spatio-temporal mutli-fractality} of the linear stochastic heat and wave equations \eqref{SHE}. Khoshnevisan, Kim, and Xiao \cite{KKX2} showed that the solution $Z^H(t,x)$ to \eqref{SHE} with space-time white noise has spatio-temporal tall peaks at infinitely many different scales (see \cite[Theorem 4.1]{KKX2}). We generalize their result into the following directions: We consider the spatio-temporal multi-fractality of \eqref{SHE} and \eqref{SWE} with colored noise in dimension $d\geq1$ (see Theorem \ref{thmsptime} below). Unlike that we only consider the spatial correlation for each fixed $t>0$ in the Theorem \ref{thmheat} and \ref{thmwave}, the situation is more complicated since the solution $Z^H(t,x)$ has a long correlation in time direction, i.e., the correlation $\Corr{\left( Z^H(t,x,) , Z^H(t,x+t^c)\right)}\rightarrow 1$ as $t\rightarrow \infty$ when $c\in (0,1/2)$ (see in \cite[Section 4]{KKX2}). Thus, to catch the tall spatio-temporal peaks, we use a \textit{stretch factor} to the time direction (see Definition \ref{stretch} below and the function $\exp(t^{2\epsilon})$ in \cite[Theorem 4.1]{KKX2} as an example). This manipulation allows us to see tall peaks in an elongated box in $(t,x)$ plane. Obviously, the stretch factor should be chosen taking into account the correlation of the solutions.
\begin{definition}\label{stretch} Let $Z(t,x)$ be the solution to \eqref{SHE} or \eqref{SWE} with a correlation measure $f$ satisfying the conditions in Theorem \ref{thmheat}. Let us set 
\begin{equation}\label{corrcondi}
    \mathcal{K}(t,x) := \Corr\left(Z(t,x), Z(t,0) \right) \quad \text{for all $x \in \R^d$ and $t>0$}. 
\end{equation} We define a {\it stretch factor} of $Z(t,x)$ by a continuous, monotonically increasing function $g:(1,\infty) \rightarrow (1,\infty)$ such that $\lim_{r\rightarrow \infty} g(r) =\infty$ and for any $\epsilon>0$,
\begin{equation}\label{corrcondi2}
    \mathcal{K}\left(g^{-1}(n), e^{n\epsilon}\right) \rightarrow 0 \quad \text{as $n\rightarrow\infty$}.
\end{equation}

\end{definition}
Now we state our last main result about the macroscopic spatio-temporal multi-fractality of \eqref{SHE} and \eqref{SWE}, which generalizes \cite[Theorem 4.1]{KKX2}.
\begin{theorem} \label{thmsptime}
Let $f$ be a correlation measure satisfying the assumptions in Theorem \ref{thmheat}. Suppose that there is a stretch factor $g$ of $Z^H(t,x)$. Then, for the random set 
\begin{equation}\label{set_heatsp}
     \mathcal{H}(\gamma) : = \left\{(e^{g(t)},x)\in (e,\infty) \times \R^d :Z^H(t,x) \geq \sqrt {2\gamma  v^H(t) g(t)  }  \right\}, \quad \text{for every $\gamma>0$},
\end{equation} we have 
\begin{equation}\label{eq_dimspheat}
    \Dimh\left(\mathcal{H}(\gamma)\right) = \left(d+1-\gamma\right) \vee d, \quad \text{a.s.}
\end{equation}

Similarly, if $g$ is a stretch factor of $Z^W(t,x)$, then for the random set 
\begin{equation}\label{set_wavesp}
    \mathcal{W}(\gamma) : = \left\{(e^{g(t)},x)\in (e,\infty)\times \R^d :Z^W(t,x) \geq \sqrt {2\gamma  v^W(t) g(t)  }  \right\},\quad \text{for every $\gamma>0$},
\end{equation} we have 
\begin{equation}\label{eq_dimspwave}
       \Dimh\left(\mathcal{W}(\gamma)\right) = \left(d+1-\gamma\right) \vee d, \quad \text{a.s.}
\end{equation}
\end{theorem}

\begin{remark}
We note that a large class of SPDEs satisfies the condition \eqref{corrcondi2}. For example, consider \eqref{SHE} with the correlation measure $f(x) = \lVert x\rVert^{-\beta}$ (i.e., the Riesz kernel \eqref{Riesz}). Then there exists a constant $c>0$ such that $\mathcal{K}(t,x) \leq c t^2 \lVert x \rVert^{-\beta} $  for all large $\lVert x \rVert $ and $t>0$ (see Lemma 3.1 in \cite{K}). Then, for any $\delta>0$, $g(r):=r^\delta $ fits into \eqref{corrcondi2}. Moreover, even if $\mathcal{K}(t,x) \leq c e^{ct} (\log \lVert x\rVert)^{-1}$ for large $\lVert x\rVert $, we can set $g(r):= e^{r}$, which gives a desired stretch factor.
\end{remark}

Now we give an outline of the paper. In Section \ref{sec2}, we present a brief explanation about the macroscopic Hausdorff dimension. In Section \ref{sec3}, we prove Theorem \ref{thmGaussian}. Section \ref{sec4} is devoted to prove the applications of Theorem \ref{thmGaussian}; Theorem \ref{thmheat}, \ref{thmwave} and \ref{thmsptime}.


\section{Macroscopic Hausdorff dimension}\label{sec2}
In this section, we introduce the notion of the macroscopic Hausdorff dimension given by Barlow and Taylor \cite{BT1,BT2}, and Khoshnevisan et al. \cite{KKX}. We also present a useful proposition that gives a lower bound for the macroscopic Hausdorff dimension.

\subsection{Definition}\label{subsec2.1}
For all integers $n\geq 1 $, we define the exponential cubes and shells as follows:
\begin{equation}
    V_n : = [e^{-n}, e^n)^d, \quad S_0 : = V_0, \quad \text{and} \quad S_{n+1}: = V_{n+1} \setminus V_n.
\end{equation}
Let $\mathcal{Q}$ be the collection of all cubes of the form
\begin{equation}
    Q(x,r) := \prod_{i=1}^d [x_i,x_i+r), 
\end{equation}
for $x=(x_1,...,x_d)\in \R^d$, and $r \in [1,\infty)$. For any subset $E\subset \R^d$, $\rho>0$, and all integers $n\geq 1$, we define 
\begin{equation}
    \nu_\rho^n(E):= \inf \left\{\sum_{i=1}^m \left (\frac{s(Q_i)}{e^n} \right)^\rho : Q_i \in \mathcal{Q}, Q_i \subset S_n \text{ and }E\cap S_n \subset \cup_{i=1}^m Q_i  \right\},
\end{equation}
where $s(Q):=r$ denotes the side of $Q=Q(x,r)$. We now introduce the definition of the macroscopic Hausdorff dimension. 
\begin{definition}[Barlow and Taylor \cite{BT1,BT2}]
The {\it macroscopic Hausdorff dimension} of $E \subset \R^d$ is defined as 
\begin{equation}
    \Dimh (E) : = \inf \left\{\rho>0: \sum_{n=1}^\infty \nu_\rho^n (E) < \infty \right\}.
\end{equation}

\end{definition}

\subsection{A lower bound for the macroscopic Hausdorff dimension}\label{subsec2.2}
Choose and fix any $\theta\in (0,1)$. We define 
$$
a_{j,n}(\theta) := e^n + je^{n\theta}, \qquad 0\leq j < e^{n(1-\theta)},
$$
$$
I_{n}(\theta) : = \bigcup_{\substack{0\leq j \leq e^{n(1-\theta)}:\\ j\in \Z}} \{ a_{j,n}(\theta)\},
$$
and
$$
\mathcal{I}_{n}(\theta) : = \prod_{i=1}^{d} I^i_{n}(\theta),
$$
where $I^i_n(\theta)$ is a copy of $I_n(\theta)$ for each $i$. We call $\cup_{n=1}^\infty\mathcal{I}_n(\theta) $ a $\theta$-skeleton of $\R^d$ (see \cite[Definition 4.2]{KKX}). Note that $\Dimh \left( \cup_{n=k}^\infty\mathcal{I}_n(\theta) \right) = d (1-\theta)$ for any integer $k\geq 1 $.
\begin{definition}[Definition 4.3, \cite{KKX}]

$E$ is called $\theta$-thick if there exists a positive integer $k=k(\theta)$ such that 
$$
E\cap Q(x,e^{n\theta}) \neq \emptyset,
$$
for all $x \in \mathcal{I}_n(\theta)$ and $n\geq k$.
\end{definition}
By the monotonicity of the macroscopic Hausdorff dimension, we get the following lower bound.
\begin{proposition}[Proposition 4.4, \cite{KKX}]\label{propthick} 
Let $E \subset \R^d$. If $E$ contains a $\theta$-thick set for some $\theta\in (0,1)$, then 
$$
\Dimh (E) \geq d(1-\theta).
$$
\end{proposition}

\section{Multi-fractality of Gaussain random fields}\label{sec3}
In this section, we prove Theorem \ref{thmGaussian}. We separate the proof into two parts: The upper bound and the lower bound of \eqref{setGaussian}.
\subsection{Proof of the upper bound in Theorem \ref{thmGaussian}}\label{subsec3.1}
For the upper bound in Theorem \ref{thmGaussian}, we use the independently discovered result of Borell (\cite{Borell}) and Tsirelson, Ibragimov, and Sudakov (TIS) (\cite{TIS}), which gives the tail probability for the supremum of Gaussian random fields. 

\begin{proposition}[Borell-TIS inequality, Theorem 2.1.3, \cite{AT}]\label{Borellineq}  Let $\{Z(t) \}_{t\in \R^d}$ be a continuous stationary Gaussian random field with mean zero and variance one. For all $x \geq \mu :=\E \left[\sup_{s\in Q(t,1)} Z(s)\right] $
$$
\P \left\{\sup_{s\in Q(t,1) }Z(s) \geq x \right\}\leq 2 e ^{-\frac{x^2}{2}+\mu x },
$$uniformly for all $t \in \R^d$.
\end{proposition} 
Note that the continuity of Gaussian random field implies that $\E \left[\sup_{s\in Q(t,1)} Z(s)\right]$ is finite (see Theorem 6.2.1, \cite{MR}). Now we are ready to present the proof of Theorem \ref{thmGaussian}. 
\begin{proof}[Proof of the upper bound in Theorem \ref{thmGaussian}] Choose and fix any $\gamma \in (0,d)$.
Recall the random set 
\begin{equation}
     \mathcal{Z}(\gamma) := \left\{ t\in \R^d  : \lVert t \rVert \geq e , Z(t) \geq  \sqrt{2\gamma \log \lVert t \rVert  }\right\}.
\end{equation}
Set, for all integers $n\geq 0 $,
\begin{equation}
     \mathcal{Z}_n(\gamma) := \left\{ t\in S_n : \lVert t \rVert \geq e , Z(t) \geq  \sqrt{2\gamma \log \lVert t\rVert }\right\},
\end{equation}
and 
\begin{equation}
    \Tilde{\mathcal{Z}}(\gamma) : = \bigcup_{n=0}^\infty \ \mathcal{Z}_n(\gamma). 
\end{equation}
Since $\mathcal{Z}(\gamma) \subset \Tilde{\mathcal{Z}}(\gamma)$, it suffices to show that $\Dimh   (\Tilde{\mathcal{Z}}(\gamma)) \leq d-\gamma$ a.s. We cover each $\mathcal{Z}_n(\gamma)$ with $Q(t,1)$ for $t\in S_n \cap \Z^d$. Then we obtain 
\begin{align*}
    \E\left[\nu^n_\rho \left( \Tilde{\mathcal{Z}}(\gamma)\right)\right] &\leq \E\left[ \sum_{t\in S_n \cap \Z^d} \left(\frac{1}{e^n}\right)^\rho \mathds{1}_{ \Tilde{\mathcal{Z}}(\gamma)}(t)\right]\\
    &\leq C e^{n(d-\rho)} \sup_{t\in S_n} \P \left\{\sup_{s\in Q(t,1)} Z(s) \geq \sqrt{2\gamma n} \right\}\\
    &\leq C e^{n(d-\gamma-\rho+o(1)) },
\end{align*}
as $n\rightarrow \infty$, where $C$ is a positive real constant which is independent of $n$. The last inequality is justified by Proposition \ref{Borellineq}. The above display shows that $ \E\left[\nu^n_\rho ( \Tilde{\mathcal{Z}}(\gamma))\right]$ is summable over all integers $n\geq 0$, thus $\sum_{n=0}^\infty\nu^n_\rho ( \Tilde{\mathcal{Z}}(\gamma))< \infty$, a.s. whenever $\rho >d-\gamma$. This implies 
$$
\Dimh (\Tilde{\mathcal{Z}}(\gamma)) \leq \rho, \quad\text{a.s.}
$$
By taking $\rho \searrow d-\gamma$, we can conclude that 
$$
\Dimh ( \Tilde{\mathcal{Z}}(\gamma)) \leq d-\gamma, \quad\text{a.s.}
$$
\end{proof}

\subsection{Proof of the lower bound in Theorem \ref{thmGaussian}}\label{subsec3.2}

For the proof of the lower bound in Theorem \ref{thmGaussian}, we use the recent result of Lopes \cite{Lopes} about the maximum of dependent Gaussian random variables. Let $X=(X_1, ..., X_n)$ be a Gaussian vector with mean vector $0$ and a covariance matrix $R$. We denote by $R_{i,j}$ an entry of $R$ for $1\leq i,j\leq n$.

\begin{proposition}[Theorem 2.2,\cite{Lopes}]\label{propgaussian} Let $X$ be a Gaussian random vector as above. Suppose that $\max_{i\neq j} R_{i,j} \leq \rho_0$ for some fixed constant $\rho\in (0,1)$. Fix another constant $\gamma_0 \in (0,1)$ and define the constants
$$
\alpha_0 = \frac{(1-\rho_0)(1-\sqrt{\gamma_0})^2}{\rho_0}, \quad \text{and} \quad \beta_0 = \frac{(1-\rho_0)(1-\sqrt{\gamma_0})}{\rho_0}.
$$
Then there exists a constant $C$ only depending on $(\gamma_0, \rho_0)$ such that 
\begin{equation}
    \P \left \{ \max_{1\leq i \leq n}X_i \leq \sqrt{2\gamma_0 (1-\rho_0)\log n }\right\}\leq C n^{-\alpha_0} \left( \log n \right)^\frac{\beta_0-1}{2}.
\end{equation}
\end{proposition}
The proof of Proposition \ref{propgaussian} is a direct calculation by constructing equally correlated Gaussian random variables and the Slepian's inequality, hence we just refer to \cite{Lopes} and begin the proof of the upper bound in Theorem \ref{thmGaussian}.
\begin{proof}[Proof of the lower bound in Theorem \ref{thmGaussian}]
Choose and fix and arbitrary $\gamma \in (0,d)$. For a technical reason, we choose another $\gamma_1\in (\gamma,d)$. We take $\theta \in (\gamma_1/d,1)$ and $\delta \in ( 0,\theta - \gamma_1/d)$ in turn. We recall the notations in Section \ref{subsec2.1}: We set, for integers $n\geq 0 $, 
$$
a_{j,n}(\theta) := e^n + je^{n\theta}, \qquad 0\leq j < e^{n(1-\theta)},
$$
$$
I_{n}(\theta) : = \bigcup_{\substack{0\leq j \leq e^{n(1-\theta)}:\\ j\in \Z}} \{ a_{j,n}(\theta)\},
$$
and
$$
\mathcal{I}_{n}(\theta) : = \prod_{i=1}^{d} I^i_{n}(\theta)
$$
where $I^i_n(\theta)$ is a copy of $I_n(\theta)$ for each $i$.  Now  choose and fix $t \in \mathcal{I}_n(\theta)$. We can find $\{ t_i\}_{i=1}^{m(n)}$ such that $t_i \in \mathcal{I}_n(\delta) \cap Q(t, e^{n\theta})$ for all $i=1,2,...,m(n)$, and  
\begin{enumerate}
    \item[(A.1)] $\lVert t_k - t_l \rVert \geq e^{n\delta}$ whenever $1\leq k<l \leq m(n)$.
    \item[(A.2)] $c^{-1}e^{nd(\theta-\delta)} \leq m(n) \leq c e^{nd(\theta-\delta)}$ for some constant $c>1$ only depending on $d$. 
\end{enumerate} 
Let us choose an arbitrary $\rho_0 \in (0, 1- \frac{\gamma_1}{d(\theta-\delta)})$. If we define 
$$
\gamma_0 : = \frac{\gamma_1}{d(\theta-\delta)(1-\rho_0)},
$$
then $0<\gamma_0 <1$ by the choice of $\rho_0$. Then by Proposition \ref{propgaussian}, we have 
\begin{align}\label{eq_propbound1}
    \P \left\{ \max_{1\leq i \leq m(n)} Z(t_i) \leq \sqrt {2\gamma_0 (1-\rho_0)\log m(n)}\right\}
    &\leq C \exp\{-\alpha_0 nd(\theta-\delta)\} \left(nd(\theta-\delta)\right)^{\frac{\beta_0-1}{2}},
\end{align}
where $C=C(d)$ is independent of $n$.
Since $m(n)$ is not equal to $e^{nd(\theta-\delta)}$, we take $\gamma_1> \gamma$ so that 
\begin{equation}
    \gamma_1> \frac{\gamma_1 \log c}{nd(\theta-\delta)} + \gamma,
\end{equation}
for all sufficiently large $n\geq 1$, where $c$ is the constant in (A.2). Clearly this is possible for an arbitrary $\gamma_1>\gamma$. 
Then we have 
\begin{align}\label{eq_propbound2}
    \P \left \{\max_{1\leq i\leq  m(n)} Z(t_i) \leq \sqrt{2\gamma n } \right\} &\leq  \P \left\{\max_{1\leq i \leq m(n)} Z(t_i) \leq \sqrt {2\gamma_1 n -\frac{ 2\gamma_1 \log c}{d(\theta- \delta)}}\right\}\nonumber \\
    &\leq \P \left\{\max_{1\leq i \leq m(n)} Z(t_i) \leq \sqrt {2\gamma_0 (1-\rho_0)\log m(n)}\right\}, 
\end{align}
for all large integers $n $.
Thus we can proceed as follows: Combining \eqref{eq_propbound1} and \eqref{eq_propbound2},
\begin{align*}
    \P\left\{ \max_{\{t_i\}_{i=1}^{m(n)} \in \mathcal{I}_n(\delta) \cap Q(t, e^{n\theta})} \frac{Z(t_i)}{\sqrt{\log \lVert t_i\rVert }} \leq \sqrt{2\gamma }\right\} &\leq \P \left\{ \max_{1\leq i \leq m(n)} \frac{Z(t_i)}{\sqrt{\log \lVert t_i\rVert }} \leq \sqrt{2\gamma}  \right\}\\
    &\leq \P \left\{ \max_{1\leq i \leq m(n)} Z(t_i) \leq \sqrt {2\gamma n }\right\}\\
    &\leq C \exp\{-\alpha_0 nd(\theta-\delta)\} \left( nd(\theta-\delta)\right)^{\frac{\beta_0-1}{2}}.
\end{align*}
Therefore, we can verify that there are many tall peaks in a macroscopic view. More precisely, we can estimate the following probability: 
\begin{align}\label{BorelGaussian}
    &\sum_{n=0}^\infty \P \left \{ \text{There exists } t\in \mathcal{I}_n(\theta) \text{ such that }  \max_{s \in \mathcal{I}_n(\delta) \cap Q(t, e^{n\theta})}\frac{Z(s)}{\sqrt{\log s}}  \leq \sqrt{2\gamma } \right\}\nonumber\\
    &\leq \sum_{n=1}^\infty \sum_{t\in \mathcal{I}_n(\theta)} \P \left\{ \max_{1\leq i \leq m(n)} Z(t_i) \leq \sqrt {2\gamma n }\right\}\nonumber\\
    &\leq \sum_{n=1}^{\infty}  C \exp\{nd \left( (1-\theta) - \alpha_0(\theta-\delta)\right) \} \left( nd(\theta-\delta)\right)^{\frac{\beta_0-1}{2}} < \infty,
\end{align}
whenever $\alpha_0 > (1-\theta)/(\theta-\delta)$. This is possible when $\rho_0$ is sufficiently small. Indeed, 
$$
\alpha_0 = \frac{1-\rho_0}{\rho_0 } \left(1- \sqrt{\frac{ \gamma_1 }{d(\theta-\delta)(1-\rho_0)}} \right)^2 
$$
goes to $\infty$ as $\rho_0$ goes to $0$ since $\frac{\gamma_1}{d(\theta- \delta)}<1$. For fixed $\theta$ and $\delta$, by the assumption on the spatial correlation, we can take any small $\rho_0$ since $\rho(\lVert t_i-t_j \rVert) \leq \rho \left(e^{n\delta}\right) \rightarrow 0 $ as $n\rightarrow \infty$. By the Borel-Cantelli lemma, \eqref{BorelGaussian} implies that the random set 
$$
 \mathcal{Z}(\gamma)= \left\{ t\in \R^d  : \lVert t \rVert \geq e , Z(t) \geq  \sqrt{2\gamma \log \lVert t\rVert }\right\}
$$
contains a $\theta$-thick set almost surely (see the proof of Theorem 4.7 in \cite{KKX} for more details). Thus $\Dimh \mathcal{Z}(\gamma)\geq d(1-\theta)$. Taking $\delta \searrow 0$ and $\theta \searrow \gamma_1 /d$ with satisfying $\theta - \delta > \gamma_1/ d $, we get the desired lower bound as we take $\gamma_1 \searrow \gamma$. 
\end{proof}

\section{Applications: Linear stochastic partial differential equations}\label{sec4}
In this section, we present some applications of Theorem \ref{thmGaussian}. First, we consider the spatial peaks of solution to \eqref{SHE} and \eqref{SWE} at each fixed time $t>0$. After that, we also give the \textit{spatio-temporal multi-fractality} of linear stochastic heat and wave equations.

\subsection{Spatial multi-fractality of linear stochastic heat and equations}\label{subsec4.1}
In this subsection, we present the proof of Theorem \ref{thmheat} and Theorem \ref{thmwave}.

\begin{proof}[Proof of Theorem \ref{thmheat}.]
By the theory of Walsh \cite{Walsh} and Dalang \cite{Dalang}, with the conditions on $f$ in Theorem \ref{thmheat}, we have the \textit{random field} solution to \eqref{SHE} 
\begin{equation}
    Z^H(t,x) = \int_0^t\int_{\R^d} p^H_{t-s}(x-y) \, F(dsdy), \quad t>0,x\in \R^d, 
\end{equation}
where $p^H(x)$ is the transition density of the isotropic $\alpha$-stable process. $\{Z^H(t,x)\}_{t\geq 0 , x\in \R^d}$ is a centered stationary Gaussian random field with finite variance (see \cite{FK}). Stationarity implies that the variance $v^H(t)$ is independent of $x$.  Moreover, the condition \eqref{contiH} guarantees the almost sure H\"older continuity of $\{Z^H(t,x)\}_{t\geq 0 , x\in \R^d}$ in both time and space variables (see \cite{BEM}). Now we claim that 
\begin{equation}\label{dimh}
    \Dimh \left(\mathcal{Z}^H_t(\gamma)\right) = (d- \gamma) \vee  0, \quad \text{a.s.}
\end{equation}
If we show that the correlation 
\begin{equation} \label{corrH}
    \Corr \left(Z^H(t,x), Z^H(t,y) \right) \rightarrow 0 \quad \text{as } \lVert x-y \rVert \rightarrow \infty \text{ for each }t>0,
\end{equation}
then the proof of the lower bound of \eqref{dimh} is a direct consequence of Theorem \ref{thmGaussian}. To this end, using the change of variables, we have 
\begin{align}\label{eq_cov}
     \E\left[Z^H(t,x)Z^H(t,y) \right] &= \int_0^t \int_{\R^d}\int_{\R^d} p^H_{t-s}(x-z)p^H_{t-s}(y-z')f(z-z')\,dzdz'ds\nonumber\\
     &=  \int_0^t \int_{\R^d}\int_{\R^d} p^H_{s}(z)p^H_{s}(z')f(z-z'+(x-y))\,dzdz'ds \nonumber\\
     &= \int_0^t \int_{\R^d} p_s^H(z)\left(p_s^H * f \right)(z+(x-y)) dz ds\nonumber\\
     &= \int_0^t \int_{\R^d} e^{-2s \lVert \xi \rVert ^\alpha+ i(x-y)\cdot \xi} \hat{f}(d\xi)\nonumber\\
     &\leq C(t) \int_{\R^d} \frac{e^{i(x-y)\cdot \xi}\hat{f}(d\xi)}{1+ \lVert \xi \rVert^\alpha}.
\end{align}
We have used the Plancherel identity and Lemma 4.1 in \cite{FK} in the last inequality. Moreover, the condition \eqref{nondegeneracy} ensures that the variance $v(t)= \E\left[Z^H(t,x)^2\right]>0$ and is actually an increasing function of $t$ (see Proposition 5.3 in \cite{CKNP2}). Hence we conclude that \eqref{corrH} holds if we apply the dominated convergence theorem to the second line with the condition (i). With the condition (iii) (or (ii) which implies (iii)) in the last line, we can obtain \eqref{corrH}. Therefore we can 
apply Theorem \ref{thmGaussian} and easily deduce \eqref{dimh}.

  For the upper bound, this is also an easy corollary of Theorem \ref{thmGaussian}, since $\{  \frac{Z^H(t,x)}{\sqrt{v^H(t)}}\}_{x\in \R^d}$ is a continuous, stationary Gaussian random field with mean zero and variance one for each $t>0$ (See section \ref{subsec3.1}). 
  \end{proof}
  

Now we prove Theorem \ref{thmwave}. We omit the details since the proof is very similar to the proof of Theorem \ref{thmheat}.
\begin{proof} [Proof of Theorem \ref{thmwave}]
Dalang \cite{Dalang} showed that, under the condition \eqref{dalang}, \eqref{SWE} has the random field solution 
 \begin{equation}
     Z^W(t,x) = \int_0^t\int_{\R^d} p^W_{t-s}(x-y) \, F(dsdy), \quad t>0, x\in \R^d, 
\end{equation}
where $p^W_t(x)$ is Green's function of the wave operator $\frac{\partial^2}{\partial t^2 }-\Delta$ in dimension $d\in \{1,2,3\}$ (see \cite{minicourse}). Under the assumptions in Theorem \ref{thmwave}, the solution $\{Z^W(t,x)\}_{t\geq 0, x \in \R^d}$ is a centered stationary continuous Gaussian random field with finite variance (for the H\"older continuity, see \cite{DS}). By the same reasons as in the derivation of \eqref{eq_cov} (also see \cite[Proposition 3.8]{FKN}), we have 
\begin{align}
 \E\left[Z^W(t,x)Z^W(t,y) \right] &= \int_0^t \int_{\R^d}\int_{\R^d} p^W_{t-s}(x-z)p^W_{t-s}(y-z')f(z-z')\,dzdz'ds\nonumber\\
     &=  \int_0^t \int_{\R^d}\int_{\R^d} p^W_{s}(z)p^W_{s}(z')f(z-z'+(x-y))\,dzdz'ds\nonumber\\
   &\leq C(t) \int_{\R^d} \frac{e^{i(x-y)\cdot \xi}\hat{f}(d\xi)}{1+ \lVert \xi \rVert^2}.
\end{align} Therefore, with a minor modification, following the similar argument as in the proof of Theorem \ref{thmheat} completes the proof.
\end{proof}

\subsection{Spatio-temporal multi-fractality of linear stochastic heat and wave equations}\label{subsec4.2}

To prove Theorem \ref{thmsptime}, we cannot use Theorem \ref{thmheat} directly, since the temporal correlation gets larger as time goes to infinity (see the paragraph below Theorem \ref{thmGaussian}). The key is to use a stretch factor $g$ in Theorem \ref{thmsptime}. This makes the temporal elapse of the solution smaller, hence the correlation on each exponential shell $S_n$ is dominated by the spatial correlation. We only prove the spatio-temporal multi-fractality for \eqref{SHE} (i.e., \eqref{eq_dimspheat}) since the proof for \eqref{SWE} is very similar.

The proof of the upper bound in Theorem \ref{thmGaussian} follows from the proper upper bound for the tail probability of Gaussian random fields.

\begin{proposition}[Proposition 4.7,\cite{KKX2}]\label{propsptime}
     
For a function $g:(1,\infty) \rightarrow (1,\infty)$ that strictly increases to infinity, for all reals $l\geq 1 $, there exists a finite positive constant $C:= C(l)>1$ such that 
\begin{align*}
    \P&\left\{ \exists t \in (b,b+l] : \sup_ {x\in Q(a,1)} \frac{Z^H(t,x)}{\sqrt{v^H(t)}} \geq \sqrt{2\gamma g(t)}\right\}\\
    &\leq 2 \exp \left\{-\gamma g(b)+ C \sqrt{2\gamma g(b)}\left( \frac{l}{\sqrt{v^H(b)}}+1\right)\right\},
\end{align*} uniformly for all $a\in \R^d$ and all sufficiently large $b\geq 1$.
\end{proposition}
\textit{Proof of Proposition \ref{propsptime}} 
The proof is very similar to the proof of \cite[Proposition 4.7]{KKX2}. The proof follows from that the condition \eqref{contiH} and \cite{BEM} implies that for all $p\geq2$ and $T>0$, there exists a positive constant $C=C(p,T,\eta)$ such that for all $0<t,s<T$ and $x,y\in\R^d$,
\begin{equation}
    \E \left[ | Z^H(t,x) - Z^H(s,y) |^p    \right] \leq C\left(|t-s|^{p(1-\eta) /2} + \lVert x-y \rVert^{p(1-\eta)/2 \wedge 1/2}  \right).
\end{equation}
The previous display, the chaining argument and Borell-TIS inequality (Proposition \ref{Borellineq}) complete the proof (see the proof of \cite[Proposition 4.7]{KKX2} for details). \qed

\begin{proof}[Proof of the upper bound in Theorem \ref{thmsptime}]

Recall the random set \eqref{set_heatsp}
\begin{align*}
    \H(\gamma) & = \left\{(e^{g(t)},x)\in (e,\infty) \times \R^d :Z^H(t,x) > \sqrt {2\gamma  v^H(t) g(t)  }  \right\}\\
    &= \left\{ (t,x) \in (1,\infty) \times \R^d : Z^H\left( g^{-1}(\log t ) , x \right)> \sqrt{2\gamma v^H( g^{-1}(\log t )) \log t } \right\}.
\end{align*}
By Proposition \ref{propsptime}, for all $a\in \R^d$ and all sufficiently large $b\geq1$ we have 
\begin{align*} 
    \P& \left\{ \exists t \in \left(g^{-1}(\log b),g^{-1}(\log (b+1) ) \right]  : \sup _{x\in Q(a,1)} Z^H(t,x) > \sqrt {2\gamma  v^H(t) g(t)  }  \right\} \\
    &\leq 2 \exp \left( -\gamma \log b + C' \sqrt{2\gamma \log(b+1) }   \right),
\end{align*}
where $C'$ is a positive and finite constant which depends only on $C$ in Proposition \ref{propsptime}.

Now choose and fix any $q>1$. Uniformly for all $a\in \R^d$, $b\in ( e^{n/q} ,e^{n+1}]$ and sufficiently large integers $n\geq 1$, 
\begin{align}\label{estiupper}
\P&\left\{(e^{g(t)},x) \in \mathcal{Z}^H(\gamma) \text{ for some }(t,x)\in (b,b+1]\times Q(a,1) \right\}\nonumber\\
&\leq 2\exp \left(-\frac{\gamma n}{q}+C' \sqrt{2\gamma (n+2)}\right).
\end{align}

Before we proceed more, let us introduce a technical lemma for calculating the macroscopic Hausdorff dimension. 
\begin{lemma}\label{lemblock}

For fixed $q>1$ and $1\leq k \leq d$, define a set $F\subset \R^{d+1}$ as

$$
F:=  \bigcup_{n=0}^\infty F_n, 
$$
where 
$$
F_n:=(0,e^{n/q}]^{k} \times (e^{n/q}, e^{n+1}]^{d+1-k}.
$$
Then $\Dimh (F)\leq d+1-k$.

\end{lemma}
\begin{proof}
Let $\kappa := d+1-k$.
We can cover each $F_n\subset S_n$ by cubes $Q\in \mathcal{Q}$ with $s(Q)=e^{n/q}$. Then the number of such cubes to cover $F_n$ is at most $Ce^{(n-n/q)\kappa}$ for some positive constant $C=C(d,k)$. Recalling the definitions in Section \ref{sec2}, 
$$
\nu_\rho^n(F) \leq c e^{n\left(\kappa(1-\frac{1}{q})+ \rho(\frac{1}{q}-1) \right)}=ce^{n \left( (1-\frac{1}{q})(\kappa -\rho)\right)},
$$
where $c>0$ is a fixed constant which is independent of $n$. Since $q>1$, and by the definition of the macroscopic Hausdorff dimension, $\Dimh (F) \leq \rho$ whenever $\rho> \kappa$. Thus, we complete the proof as $\rho$ decreases to $\kappa$. 
\end{proof}

Now choose and fix any $\epsilon := (\epsilon_1,...,\epsilon_d) \in \{0,1\}^d$. Let us define an orthant as 
\begin{equation*}
\mathcal{O}_\epsilon : =\left\{x= (x_1,x_2,...,x_d) \,: \, (-1)^{\epsilon_1} x_1>0, (-1)^{\epsilon_2} x_2>0,...,(-1)^{\epsilon_d} x_d>0 \right\}.
\end{equation*} We also define 
\begin{equation*}
\H_\epsilon:=\H_\epsilon(\gamma) := \H(\gamma) \cap \mathcal{O}_\epsilon.
\end{equation*} Then, by a symmetric argument, to prove 
\begin{equation}\label{dimepsilon}
\Dimh (\H_\epsilon) \leq (d+1-\gamma) \vee d, \quad \text{a.s.,}
\end{equation} for every $\epsilon \in \{0,1\}^d$, it suffices to show 
\begin{equation}
\Dimh(\H_{\epsilon_0}) \leq (d+1-\gamma) \vee d, \quad \text{a.s.,}
\end{equation} where $\epsilon_0 := (0,...,0) \in \{0,1\}^d$. Let us set 
\begin{equation*}
\mathcal{L}_n := \mathcal{L}_n(q,n,\gamma) := \H_{\epsilon_0}\cap \left( e^{n/q}, e^{n+1}\right]^{d+1},
\end{equation*} for all integers $n\geq 1$. Then one can see that by Lemma \ref{lemblock}, 
\begin{equation}
\Dimh \left( \H_{\epsilon_0} \setminus \bigcup_{n=1}^\infty \mathcal{L}_n\right) \leq d.
\end{equation} Furthermore, the property of the macroscopic Hausdorff dimension implies that 
\begin{align}\label{dimbound}
 \Dimh(\H_{\epsilon_0}) \leq \Dimh \left( \H_{\epsilon_0} \setminus \bigcup_{n=1}^\infty \mathcal{L}_n\right) \vee \Dimh \left(\bigcup_{n=1}^\infty \mathcal{L}_n\right).
\end{align} Next we claim that 
\begin{equation}\label{dimclaim}
\Dimh \left(\bigcup_{n=1}^\infty \mathcal{L}_n\right) \leq \left( d+1- \frac{\gamma}{q}\right) \vee d, \quad \text{a.s.}
\end{equation} To get the desired upper bound, we cover each $\mathcal{L}_n$ with the squares of the form $(b,b+1] \times Q(a,1)$ which satisfy 
\begin{equation}\label{covering}
    \sup_{x\in Q(a,1)} Z^H( g^{-1}(\log t), x) > \sqrt{2\gamma v^H( g^{-1}(\log t)) \log t } \quad \text{for some }t\in (b,b+1].
\end{equation} Then, equipped with the estimation \eqref{estiupper}, we have 
\begin{align*}
    \E\left[ \nu_\rho^n( \mathcal{L}_n)\right] &\leq \E\left[\sum_{\substack{(b,b+1]\times Q(a,1) \subset (e^{q/n},e^{n+1}]^{d+1}:\\\eqref{covering} \text{ holds}}} \frac{1}{e^{n\rho}}\right]\\
    &\leq C \cdot \exp \left\{-n\left(\frac{\gamma}{q}+\rho-(d+1)+ O\left(\frac{1}{\sqrt{n}} \right) \right) \right\},
\end{align*}as $n \rightarrow\infty$, where the constant $C>0$ is independent of $n$. The above display ensures that 
\begin{equation*}
\sum_{n=1}^\infty \nu_\rho^n(\mathcal{L}_n) < \infty \quad \text{a.s., for every $\rho \in \left( d+1 - \frac{\gamma}{q} ,d+1 \right]$},
\end{equation*} provided that $d+1 > \gamma/q$. Then the definition of the macroscopic Hausdorff dimension implies that \eqref{dimclaim} holds. Since $q>1$ is arbitrary (note that the definition of $\H_{\epsilon_0}$ does not depend on $q$), \eqref{dimbound} gives 
\begin{equation*}
 \Dimh(\H_{\epsilon_0}) \leq (d+1-\gamma) \vee d, \quad \text{a.s.,}
\end{equation*} which proves \eqref{dimepsilon}. Moreover, it is clear that 
\begin{equation*}
\Dimh( (1,\infty) \times \{0\} ) =1.
\end{equation*} This and \eqref{dimepsilon} together imply that 
\begin{equation}
\Dimh (\H(\gamma)) \leq \max_{\epsilon \in \{0,1\}^d} \left\{\Dimh (\H_\epsilon) \right\} \leq (d+1-\gamma) \vee d, \quad \text{a.s.}
\end{equation} This establish the desired upper bound. 
\end{proof}

We now prove the lower bound in Theorem \ref{thmsptime}, using Proposition \ref{propgaussian} and Theroem \ref{thmheat}.

\begin{proof}[Proof of the lower bound in Theorem \ref{thmsptime}]
Choose and fix any $\gamma \in (0,d)$. Let $\gamma_1, \gamma_2$ be fixed constants such that 
\begin{equation}\label{gammaorder}
    0<\gamma< \gamma_1 <\gamma_2 < d.
\end{equation}
We fix $\theta \in (\frac{\gamma_2}{d} ,1)$ and $\delta \in (0, \theta - \frac{\gamma_2}{d})$. Take $\rho_0\in \left(0,1-\frac{\gamma_2}{d(\theta-\delta)(1-\rho_0)} \right)$. Define 
\begin{equation}
    \gamma_0 := \frac{\gamma_2}{d(\theta-\delta)(1-\rho_0)}.
\end{equation}
According to the notations in Section \ref{subsec3.2}, for any $x\in \mathcal{I}_n(\theta)$, we can find $\{x_i\}_{i=1}^{m(n)}$ such that $x_i \in \mathcal{I}_n(\delta) \cap Q(x,e^{n\theta})$ for all $i=1,...,m(n)$, and 
\begin{enumerate}
    \item[(B.1)] $\lVert x_k - x_l \rVert \geq e^{n\delta}$ whenever $1\leq k<l \leq m(n)$.
    \item[(B.2)] $c^{-1}e^{nd(\theta-\delta)} \leq m(n) \leq c e^{nd(\theta-\delta)}$ for some constant $c>1$ only depending on $d$. 
\end{enumerate} 
Note that for all sufficiently large integers $n\geq 1$,
\begin{equation}\label{gamma1}
    \sup_{t\in(e^n,e^{n+1}]}\left( \gamma \cdot \frac{\log t}{n}\right) <\gamma_1,
\end{equation}
and
\begin{equation}\label{gamma2}
    \gamma_2 > \frac{\gamma_2\log c}{nd(\theta-\delta)}+\gamma_1,
\end{equation} where $c$ is the constant in (B.2). Note that \eqref{gamma1} and \eqref{gamma2} hold for arbitrary choices of $\gamma, \gamma_1$, and $\gamma_2$ satisfying \eqref{gammaorder}, for all sufficiently large $n\geq 1 $.
With the similar procedure as in Section \ref{subsec3.2}, by (B.2) and \eqref{gamma2}, we have 
\begin{align*}
    &\sup_{t\in (e^n,e^{n+1}]} \max_{x\in \mathcal{I}_n(\theta)} \P \left\{ \max_{\{x_i\}_{i=1}^{m(n)} \in \mathcal{I}_n(\delta) \cap Q(x,e^{n\theta})}Z^H( g^{-1}(\log t ), x_i) \leq \sqrt{2v^H(g^{-1}(\log t ))\gamma \log t} \right\}\\
    &\leq \sup_{t\in (e^n,e^{n+1}]} \max_{x\in \mathcal{I}_n(\theta)} \P \left\{\max_{\{x_i\}_{i=1}^{m(n)} \in \mathcal{I}_n(\delta) \cap Q(x,e^{n\theta})}Z^H( g^{-1}(\log t ), x_i) \leq \sqrt{2v^H(g^{-1}(\log t ))\gamma_1 n}  \right\}\\
    &\leq \sup_{t\in (e^n,e^{n+1}]} \max_{x\in \mathcal{I}_n(\theta)} \P \left\{\max_{\{x_i\}_{i=1}^{m(n)} \in \mathcal{I}_n(\delta) \cap Q(x,e^{n\theta})}Z^H( g^{-1}(\log t ), x_i) \leq \sqrt{2v^H(g^{-1}(\log t ))\left(\gamma_2 n - \frac{\gamma_2 \log c }{d(\theta-\delta)} \right)}  \right\}\\
    &\leq \sup_{t\in (e^n,e^{n+1}]} \max_{x\in \mathcal{I}_n(\theta)} \P \left\{ \max_{1\leq i \leq m(n)} \frac{Z^H( g^{-1}(\log t ), x_i)}{\sqrt{v^H(g^{-1}(\log t ))}}\leq \sqrt{2\gamma_0(1-\rho_0) \log m(n) }\right\}\\
    &\leq C \exp \left(-\alpha_0 nd(\theta-\delta) \right) (nd(\theta-\delta))^{\frac{\beta_0-1}{2}},
\end{align*}
for all large integers $n\geq 1 $. The last inequality is justified by the condition \eqref{corrcondi2} and Proposition \ref{propgaussian}: If we set $\mathcal{K}(t,x) : = \Corr\left( Z^H(t,x) , Z^H(t,0)\right)$ then due to the stationarity we get 
\begin{align*}
    \sup_{t\in (e^n, e^{n+1}]} \sup_ {\lVert x-y\rVert \geq e^{nd(\theta-\delta)}}& \Corr\left( Z^H(g^{-1}(\log t), x),Z^H(g^{-1}(\log t), y)\right) \\&\leq C \cdot  \mathcal{K} ( g^{-1}(n), e^{nd(\theta-\delta)}) \rightarrow 0 \quad \text{as } n \rightarrow \infty,
\end{align*} where $C=C(a,d,\theta,\delta)>0$ is a fixed constant. Thus we can choose an arbitrary small $\rho_0>0$ which satisfies
\begin{equation}\label{rhocondi}
\sup_{t\in (e^n, e^{n+1}]} \sup_ {\lVert x-y\rVert \geq e^{nd(\theta-\delta)}} \Corr\left( Z^H(g^{-1}(\log t), x),Z^H(g^{-1}(\log t), y)\right) < \rho_0,
\end{equation} for all sufficiently large $n\geq 1$.

Let 
$$
\mathcal{H}_0(\gamma) : = \mathcal{H}(\gamma)\cap \bigcup_{n=0}^\infty \left( e^n , e^{n+1}\right]^{d+1}.
$$
Then with Proposition \ref{propgaussian} and \eqref{rhocondi}, we can deduce that, for all large $n\geq 1$ there exists a constant $C:=C(d, \theta, \delta)>0$ such that 
\begin{align*}
    \P&\left\{  \mathcal{H}_0(\gamma) \cap \left[\{ t\} \times Q(x,e^{n\theta}) \right]= \varnothing \text{ for some }t \in \Z \cap (e^n , e^{n+1}] \text{ and }x \in \mathcal{I}_n(\theta)\right\}\\
    &\leq\P \left\{ \min_{\substack{t\in \Z\\t\in (e^n,e^{n+1}]}} \min_{x\in \mathcal{I}_n(\theta)} \max_{\{x_i\}_{i=1}^{m(n)} \in \mathcal{I}_n(\delta) \cap Q(x,e^{n\theta})}\frac{Z^H( g^{-1}(\log t ), x_i)}{\sqrt{v^H(g^{-1}(\log t ))}}\leq \sqrt{2\gamma_0(1-\rho_0) \log m(n) }\right\}\\
    &\leq C \cdot \exp\left\{nd(1- \theta) + n -\alpha_0 nd(\theta-\delta)\right\}\left(nd(\theta-\delta) \right)^{\frac{\beta_0-1}{2}}.
\end{align*}
As in Section \ref{subsec3.2}, we get $\alpha_0 \rightarrow \infty $ as $\rho_0 \rightarrow 0$. With sufficiently small $\rho_0$,
the Borel-Cantelli lemma implies the following:
\begin{equation}\label{assertion}
  \mathcal{H}_0(\gamma) \cap \left[\{ t\} \times Q(x,e^{n\theta}) \right]\neq \varnothing \text{ for all }t \in \Z \cap (e^n , e^{n+1}] \text{ and }x \in \mathcal{I}_n(\theta)
\end{equation}
for all sufficiently large integers $n\geq 1 $, almost surely. The remaining things are very similar to the proof of \cite[Theorem 3.10]{KKX2}. But since we deal with a general dimension $d\geq1$, we present the remaining part of the proof for the sake of completeness, following the proof of \cite[Theorem 3.10]{KKX2}.

Let us set 
\begin{equation*}
A_{\Vec{j},n}(\theta) := (a_{j_1,n}(\theta) , a_{j_1+1,n}(\theta)] \times \cdots \times (a_{j_d,n}(\theta) , a_{j_d+1,n}(\theta)], 
\end{equation*} and
\begin{equation*}
\eta_{\vj,n}(t) := \inf\left\{ x\in A_{\vj,n}(\theta) \, :\, (t,x) \in \H_0(\gamma) \right\}\footnote{For a set $E\in\R^d$, we define $\inf(A) := x_0$ where $x_0$ satisfies $\min_{x\in A} \lVert x\rVert = x_0$. We also define $\inf(\varnothing) := \infty$. },
\end{equation*}  for every $\vj= (j_1,...,j_d) \in [0, e^{n(1-\theta)})^d \cap \Z^d$, all integers $t\in (e^n, e^{n+1}]$ and sufficiently large $n\geq 1$.  Note that \eqref{assertion} guarantees that $\eta_{\vj,n}(t)$ is a well-defined random variable. Now we define a random Borel measure $\mu$ on $(1,\infty) \times \R^d$ on each shell $S_n$ as follows: For all integers $n \geq 0$ and all Borel sets $A \subset (1,\infty) \times \R^d$,
\begin{equation*}
\mu(A \cap S_n) := \sum_{\substack{t\in \Z\\t\in (e^n,e^{n+1}]}}  \sum_{\substack{\vj\in \Z^d\\\vj\in [0,e^{n(1-\theta)})^d}} {\mathds{1}}_A(t, \eta_{\vj,n}(t) ),
\end{equation*} where ${\mathds{1}}_A(t,x)$ denotes the indicator function for $(t,x) \in (1,\infty)\times \R^d$. One can observe that \eqref{assertion} implies 
\begin{equation}\label{mulower}
\mu(S_n) \geq \frac{1}{2} e^{nd(1-\theta)+n}, \quad \text{a.s.,}
\end{equation} for all large integers $n\geq 1$. 

Now we fix a large integer $n$ such that \eqref{mulower} holds. Choose and fix an arbitrary integer $s\in (e^n, e^{n+1}], x \in \R^d$ and $r\geq1$ which satisfy $Q(x,r) \subset (e^n,e^{n+1}]^d$. There are two cases: (i) $r \leq e^{n\theta}$ and (ii) $e^{n\theta} \leq r \leq e^{n+1} -e^n $. We deal with two cases separately.

For the case (i), it is easy to see that there are at most $2^d$ vectors $\vj$ such that $\eta_{\vj,n}(s) \in Q(x,r)$. Therefore, $\mu( \{s\} \times Q(x,r) ) \leq 2^d \leq 2^d r^a$ for every $a>0$. Summing over all integers $s \in (t,t+r]$, we have
\begin{equation}\label{eq_r1}
\mu((t,t+r] \times Q(x,r)) \leq 2^d r^{1+a},
\end{equation} provided that $(t,t+r] \times Q(x,r) \subset S_n$ and $a>0$.

For the case (ii),  there are at most $ (1+re^{-n\theta})^d\leq (2r)^d e^{-nd\theta}$ vectors $\vj$ such that $\eta_{\vj,n}(s) \in Q(x,r)$. Therefore, for every integers $s\in(t,t+r]$ and all cubes $Q(x,r)\subset (e^n,e^{n+1}]^d$, we have 
\begin{align*}
\mu(\{s\}\times Q(x,r) )& \leq 2^d r^d e^{-nd\theta} \leq 2^d e^{-nd\theta} r^a \sup_{e^{n\theta} < r \leq e^{n+1}-e^n} r^{d-a}\\
&\leq 2^d e^{n(d-a-d\theta)} (e-1)^{d-a} r^a,
\end{align*} for every $a>0$. In particular, we set $a:=d-d\theta$ and sum the last inequality over all integers $s\in (t,t+r]$ to see
\begin{equation}\label{eq_r2}
\mu((t,t+r]\times Q(x,r) ) \leq 2^d (e-1)^{d\theta} r^{d+1-d\theta},
\end{equation} whenever $(t,t+r] \times Q(x,r) \subset S_n$. From \eqref{eq_r1} and \eqref{eq_r2}, in both cases (i) and (ii), we have 
\begin{equation}\label{eq_r3}
\frac{\mu((t,t+r] \times Q(x,r) )}{r^{d+1}} \leq 2^d (e-1)^d=:c_d, 
\end{equation} for all cubes $(t,t+r] \times Q(x,r) \subset S_n$ and $r\geq 1 $ (we also used the fact that $0<\theta<1$). \eqref{eq_r3} and the density theorem of Barlow and Taylor \cite[Theorem 4.1]{BT2}, we have 
\begin{equation*}
\nu_{d+1-d\theta}^n(\H_0(\gamma)) \geq c_d^{-1} e^{-n(d+1-d\theta)} \mu(S_n) \geq 2^{-1} c_d^{-1}, \quad \text{a.s.,}
\end{equation*} for all sufficiently large integers $n\geq1$. The last inequality is justified by \eqref{mulower}. We sum over $n$ to obtain 
\begin{equation*}
\sum_{n=1}^\infty \nu_{d+1-d\theta}^n(\H_0(\gamma)) = \infty, \quad \text{a.s.}
\end{equation*} Recall that $\theta \in (\gamma_2/d ,1)$, $\delta\in(0,\theta-\gamma_2/d)$ and \eqref{gammaorder}, i.e., $0<\gamma<\gamma_1<\gamma_2<d$. With satisfying this order of constants, we let $\delta \searrow 0 $ and $\theta \searrow \gamma_2/d$. Then we have 
\begin{equation*}
\Dimh(\H(\gamma))\geq \Dimh(\H_0(\gamma)) \geq d+1-\gamma_2, \quad \text{a.s.}
\end{equation*} Since we can take $\gamma_1$ and $\gamma_2$ arbitrary close to $\gamma$ (see \eqref{gamma1} and \eqref{gamma2}) without violating $\gamma_1>\gamma_2$, we obtain that
\begin{equation}\label{eq_dimlower1}
\Dimh(\H_0(\gamma)) \geq d+1-\gamma, \quad \text{a.s.}
\end{equation} 

Now it remains to verify that 
\begin{equation}\label{eq_dimlower2}
\Dimh(\H(\gamma)) \geq d, \quad \text{a.s.}
\end{equation}  Let us choose and fix a real $t_0>e$ such that $g^{-1}(\log t_0) >0$. We claim that 
\begin{equation}\label{eq_dimlower3}
\Dimh \left\{\lVert x \rVert \geq M \,:\, \frac{Z^H(g^{-1}(\log t_0),x)} {v^H(g^{-1}(\log t_0))}   \geq \sqrt{2\gamma \log t_0}   \right\} =d, \quad \text{a.s.,}
\end{equation} for every $\gamma>0$. Let $\gamma_0$ be a fixed real constant in $(0,d)$. Note that the value of the left hand side in \eqref{eq_dimlower3} is independent of $M$ by the definition of the macroscopic Hausdorff dimension. With this in mind, choose large $M>0$ so that $\gamma_0 \log M \geq \gamma$. Then we have 
\begin{align*}
\Dimh &\left\{\lVert x \rVert \geq M \,:\, \frac{Z^H(g^{-1}(\log t_0),x)} {v^H(g^{-1}(\log t_0))}   \geq \sqrt{2\gamma \log t_0}   \right\} \\
&\geq\Dimh \left\{\lVert x \rVert \geq M \,:\, \frac{Z^H(g^{-1}(\log t_0),x)} {v^H(g^{-1}(\log t_0))}   \geq \sqrt{2\gamma_0 \log M}   \right\} \\
&\geq\Dimh \left\{\lVert x \rVert \geq M \,:\, \frac{Z^H(g^{-1}(\log t_0),x)} {v^H(g^{-1}(\log t_0))}   \geq \sqrt{2\gamma_0 \log \lVert x \rVert }   \right\} \\
&\geq d-\gamma_0, \quad \text{a.s.,}
\end{align*} by Theorem \ref{thmheat}. Therefore, as $\gamma_0 \searrow 0$, we get \eqref{eq_dimlower3}. Moreover, \eqref{eq_dimlower3} easily implies \eqref{eq_dimlower2} (see \cite[Remark 3.12]{KKX2}), which completes the proof.

\end{proof}

\section*{Acknowledgments} The author would like to thank Professor Kunwoo Kim for his helpful comments and continued support. 

\begin{spacing}{1}
\begin{small}
\end{small}\end{spacing}
\vskip.1in

\begin{small}
\noindent\textbf{Jaeyun Yi}\\
\noindent Department of Mathematics, Pohang University of Science and Technology (POSTECH), Pohang, Gyeongbuk, Korea 37673\\
\noindent \emph{Email address}: \texttt{stork@postech.ac.kr}
\end{small}

\end{document}